\documentclass[11pt]{article}
\usepackage[left=1in,top=1in,right=1in,bottom=1in]{geometry}
\usepackage{amssymb, amsmath, amsthm, graphicx,mathrsfs}
\usepackage[doi=false,isbn=false,url=false,eprint=false]{biblatex}
\usepackage{tikz}
\usetikzlibrary{fadings}
\def\qed{\hfill\ifhmode\unskip\nobreak\fi\quad\ifmmode\Box\else\hfill$\Box$\fi}
\def\ite#1{\hfill\break${}$\hbox to 50pt {\quad(#1)\hfill}}
\newtheorem{thm}{Theorem}[section]
\newtheorem{cor}[thm]{Corollary}

\newtheorem{lem}[thm]{Lemma}

\newtheorem{claim}[thm]{Claim}

\theoremstyle{definition}

\def\cA{{\mathcal A}}
\def\cB{{\mathcal B}}
\def\cC{{\mathcal C}}

\def\cF{{\mathcal F}}
\def\cG{{\mathcal G}}
\def\cH{{\mathcal H}}
\def\cL{{\mathcal L}}

\def\cN{{\mathcal N}}
\begin{document}

\title{\vspace{-0.5in}
 Berge Pancyclic hypergraphs}

\author{
{{Teegan Bailey}}\thanks{
\footnotesize {Email: {\tt teeganb@email.sc.edu}.
}}
\and{{Yupei Li}}\thanks{
\footnotesize {Email: {\tt yupei@email.sc.edu}.
}}
\and{{Ruth Luo}}\thanks{
\footnotesize {E-mail: {\tt ruthluo@sc.edu}. University of South Carolina, SC 29208
}}}
\date{\today}

\maketitle

\vspace{-0.3in}

\begin{abstract}
An $n$-vertex graph $G$ is {\em pancyclic} if for every $\ell$, $3\leq \ell \leq n$, $G$ contains a cycle of length $\ell$. We consider an analog for hypergraphs. An $r$-uniform hypergraph is a set of vertices and a set of edges in which every edge contains $r$ vertices. A Berge cycle of length $\ell$ in a hypergraph is an alternating sequence of $\ell$ distinct vertices and $\ell$ distinct edges $v_1,e_1,v_2, \ldots, v_\ell, e_{\ell},v_1$ such that $\{v_i, v_{i+1}\} \subseteq e_i$ for all $i$, with indices taken modulo $\ell$.
We similarly call a hypergraph {\em Berge pancyclic} if it contains Berge cycles of every possible length. 

For sufficiently large $n$ we prove a sharp minimum degree condition guaranteeing Berge pancyclicity in $r$-uniform hypergraphs where $3 \leq r \leq \lfloor (n-1)/2\rfloor - 2$. This bound is optimal for all such combinations of $n$ and $r$.

\end{abstract}

\section{Introduction}
A graph $G$ is a set of vertices $V(G)$ and a set of edges $E(G)$ in which every edge is a subset of $2$ vertices. The {\em neighborhood} of a vertex $v$, denoted $N_G(v)$, is the set of vertices that are adjacent to $v$. The size of the neighborhood of $v$ is called the {\em degree} of a vertex $v$, $d_G(v)$. The {\em minimum degree} of $G$ is defined $\delta(G) = \min_{v \in V(G)} d_G(v)$. When there is no room for ambiguity, we may drop the subscripts and simply write $N(v)$ and $d(v)$. 

 We say $G$ is \textbf{hamiltonian} if it contains a cycle $C$ such that $V(C) = V(G)$, i.e. $C$ spans the vertices of $G$. The \textbf{circumference} of $G$, denoted $c(G)$, is length of the longest cycle contained in $G$. Hence if $G$ is an $n$-vertex hamiltonian graph, then $c(G) = n$. Hamiltonian graphs form one of the most well studied areas of graph theory and have given us several celebrated results such as the following result from Dirac.

    \begin{thm}[Dirac~\cite{Dirac}]\label{D} Let $n \geq 3$ and let $G$ be an $n$-vertex graph with minimum degree $\delta(G)$. If $\delta(G) \geq \frac n2$, then $G$ is hamiltonian.
    \end{thm}
    The bound $\delta(G)\geq \frac{n}{2}$ is best possible. For $n \geq 3$, consider the graph $G$ composed of two cliques of size $\lceil \frac{n}{2}\rceil$ which share one vertex when $n$ is odd and no vertices when $n$ is even. Then $G$ contains no hamiltonian cycle because it can be disconnected by removing at most one vertex. Another nonhamiltonian extremal example is the complete bipartite graph $K_{\lfloor (n-1)/2 \rfloor, \lceil (n+1)/2 \rceil}$. In both constructions the minimum degree is $\lfloor(n-1)/2 \rfloor$. 
    
    Bounds on the minimum degree of a graph which force the existence of some structure are referred to as \textit{Dirac-type} conditions.
    
    Ore later proved a strengthening of Dirac's Theorem involving the degree sum of nonadjacent vertices. The main idea of Ore's proof is to consider the end vertices of a longest path. If these vertices have a sufficiently large degree sum, then one can find either a longer path or a hamiltonian cycle. 

    \begin{thm}[Ore~\cite{Ore}]\label{Ore}
        Let $G$ be an $n$-vertex graph. If $d_G(u) + d_G(v) \geq n$ for every pair of nonadjacent vertices $u,v\in V(G)$, then $G$ is hamiltonian.
    \end{thm}

    The property of hamiltonicity can be generalized by asking when a graph $G$ is guaranteed to contain a cycle of length $\ell$ for some $3 \leq \ell \leq n$. A much stronger property is  when $G$ contain cycles of every length $\ell$ for $3\leq \ell \leq n$. An $n$-vertex graph $G$ is \textbf{pancyclic} if it contains cycles of all lengths $\ell$, $3 \leq \ell \leq n$. A preliminary result for pancyclic graphs is due to Bondy.

    \begin{thm}[Bondy~\cite{Bondy}]\label{Bondy}
        Let $G$ be an $n$-vertex hamiltonian graph with $|E(G)| \geq \frac {n^2}4$. Then either $G$ is pancyclic or $K_{\frac n2, \frac n2}$.
    \end{thm}

    This theorem asserts that the same conditions in Theorem~\ref{D} not only guarantees a cycle of length $n$, but also cycles of {\em every} length unless the graph is one specific exceptional graph. In~\cite{Bondy}, Bondy proposed a so-called meta-conjecture: {\em ``almost any
    nontrivial condition on a graph which implies that the graph is hamiltonian also implies that the
    graph is pancyclic.” } He later added {\em ``...there may be a simple family of exceptional graphs."}

    Since then, many pancyclic results in support of Bondy's meta-conjecture have appeared. For example, Letzter~\cite{Letzter} proved a conjecture of Jackson and Ordaz that states that graphs with vertex-connectivity larger than their independence number are pancyclic. This strengthens a well-known result by Chv\'atal and Erd\H{o}s.  Another example is Fleischner's Theorem~\cite{Fleischner} which states that the square of a graph is pancyclic if and only if it is hamiltonian.~\footnote{The square of a graph $G$, denoted $G^2$, is the graph obtained from $G$ by adding edges between vertices of distance $2$.} See example~\cite{Draganic,Cream, Schmeichel} for some more results related to Bondy's meta-conjecture. 

    A weaker version of pancyclicity  involves the \textbf{girth} of a graph $G$, which is the length of the shortest cycle contained in $G$, denoted $g(G)$. We say a graph $G$ is \textbf{weakly pancyclic} if $G$ contains cycles of all lengths $\ell$ where $g(G) \leq \ell \leq c(G)$. Note that if $g(G) = 3$ and $c(G) = n$, then being weakly pancyclic is equivalent to being pancyclic. For a more extensive list of results about weakly pancyclic graphs see~\cite{BFG}. We highlight the following which we will use to prove our main result.

    \begin{thm}[Brandt~\cite{Brandt}]\label{Brandt}
        Every nonbipartite $n$-vertex graph $G$ with minimum degree $\frac{n+2}{3}$ is weakly pancylic, with girth either 3 or 4.
    \end{thm}

    Our main result is a Dirac-type result for pancyclicity of uniform hypergraphs. This can be viewed as extending Bondy's meta-conjecture to hypergraphs. Before doing so, we present several definitions and results for hypergraphs.
    A \textbf{hypergraph} $\cH$ is a set of vertices $V(\cH)$ and hyperedges $E(\cH)$, where $e \subseteq V(\cH)$ for each $e \in E(\cH)$.  
    If $|e| = r$ for every $e\in E(\cH)$, then we say that $\cH$ is an \textbf{ $r$-uniform hypergraph}, or simply an {\bf $r$-graph}. Note that in the case when $\cH$ is $2$-uniform, we recover the traditional definition of a graph.
    The {\bf degree} of a vertex $v \in V(\cH)$, denoted $d_\cH(v)$, is the number of hyperedges in $\cH$ containing $v$. We define the {\bf minimum degree} of $\cH$ as $\delta(\cH) = \min_{v \in V(\cH)} d_\cH(v)$.

    There are several ways to define a cycle in a hypergraph; our result focuses on Berge cycles. A \textbf{Berge cycle} is a set of $\ell$ distinct vertices and $\ell$ distinct hyperedges, say $\{v_1, ..., v_\ell\}$ and $\{e_1, ..., e_\ell\}$, such that $\{v_i,v_{i+1}\} \subset e_i$, where the indices are taken modulo $\ell$. Hence we may write a Berge cycle $\cC$ as $v_1,e_1,v_2,e_2,\ldots, v_\ell,e_{\ell},v_1$.
    We say a {\bf Berge hamiltonian cycle} in a hypergraph $\cH$ is a Berge cycle of length $n = |V(\cH)|$, and $\cH$ is {\bf pancyclic} if it contains Berge cycles of every length between $2$ and $n$. (Note that Berge cycles may have length $2$ in hypergraphs with uniformity $r \geq 3$.)

    There are several results centered around Berge cycles in hypergraphs. For instance, Gy\H{o}ri and Lemons~\cite{Gyori_Lemons} studied the maximum number of hyperedges in uniform hypergraphs that avoid Berge cycles of some fixed length; F\"uredi, Kostochka, and the last author~\cite{FKL1, FKL2} studied the analogous problem for hypergraphs with bounded circumference; Jiang and Ma~\cite{JM} proved sufficient average degree conditions for the existence of $k$-consecutive Berge cycle lengths; and Clemens,
Ehrenm\"uller and Person~\cite{CEP}, Ma, Hou, and Gao~\cite{MHG}, and Coulson and Perarnau~\cite{Coulson} proved early Dirac-type results for the existence of hamiltonian Berge cycles. 

We highlight below two sharp Dirac-type results for uniform and non-uniform hypergraphs. See~\cite{Gyori, Gyori_Lemons,GLSZ, Lu, Coulson, KLM, KLM2, MHG, KL, FKL1, FKL2, Salia} for more results related to hamiltonian and long Berge cycles in hypergraphs.

    \begin{thm}[Kostochka, Luo, McCourt~\cite{KLM}]\label{KLMthm}
        Let $3 \leq  r <n$, and let $H$ be an $n$-vertex $r$-uniform hypergraph. If
        (a) $r\leq \left\lfloor (n-1)/2 \right\rfloor$ and 
        $\delta(H) \geq  \binom{\left\lfloor (n-1)/2 \right\rfloor}{r-1} +1$ or 
        (b) $r\geq n/2$ and $\delta(H) \geq r$, 
        then $H$ contains a hamiltonian Berge cycle. 
    \end{thm}

    \begin{thm}[F\"uredi, Kostochka, Luo~\cite{FKL1}]\label{FKLthm}
        Let $n\geq 15$ and let $\cH$ be an $n$-vertex (not necessarily uniform) hypergraph such that $\delta(\cH) \geq 2^{(n-1)/2}+1$ if $n$ is odd or $\delta(\cH) \geq 2^{n/2-1}+2$ if $n$ is even. Then $\cH$ contains a hamiltonian Berge cycle.
    \end{thm}
    
     Observe that in Theorem~\ref{KLMthm} the minimum degree condition changes dramatically depending on the uniformity of the hypergraph.
     
    Minimum degree bounds in both Theorems~\ref{KLMthm} and~\ref{FKLthm} are best possible. In particular, for Theorem~\ref{KLMthm} we have the following constructions with minimum degree $\binom{\lfloor (n-1)/2 \rfloor}{ r-1}$, both of which come from \cite{KLM}.
    
    {\bf Construction 1.} Let $3 \leq r \leq \lfloor (n-1)/2 \rfloor -1$. If $n$ is odd, then $\cH_1$ is the $n$-vertex $r$-uniform hypergraph consisting of two $(n+1)/2$-cliques sharing exactly one vertex. If $n$ is even, then $\cH_1$ is the $n$-vertex $r$-uniform hypergraph consisting of two disjoint $n/2$-cliques and one additional hyperedge. 

    {\bf Construction 2.}  Let $3 \leq r \leq \lfloor (n-1)/2 \rfloor -1$. The vertex set of the $n$-vertex $r$-uniform hypergraph $\cH_2$ consists of two disjoint sets of size $\lfloor (n-1)/2 \rfloor$ and $\lceil (n+1)/2 \rceil$, respectively. The set of hyperedges of $\cH_2$ consists of all $r$-edges which contain at most one vertex from the set of size $\lceil (n+1)/2 \rceil$.
    
    These constructions are natural hypergraph analogs of the extremal examples for Dirac's Theorem, Theorem~\ref{D}. Construction 1 does not have a hamiltonian Berge cycle because the removal of a single vertex or a single hyperedge disconnects the hypergraph. A hamiltonian Berge cycle in Construction 2 would need to have a pair of consecutive vertices in the larger set, but no hyperedge can join these vertices.

\section{Main result}
    Our main result is a strengthening of Theorem~\ref{KLMthm} for small uniformity---we show that the same minimum degree condition that guarantees the existence of a hamiltonian Berge cycle in fact is sufficient for the much stronger condition of pancyclicity whenever the hypergraph is sufficiently large. 

    \begin{thm}\label{main}
        Let n and r be positive integers such that $3 \leq r \leq \lfloor \frac{n-1}{2} \rfloor - 2$, $n \geq 70$ if $r \in \{3,4\}$, and $n \geq 39$ if $r \geq 5$. If $\cH$ is an $n$-vertex, $r$-uniform hypergraph with $\delta(\cH) \geq \binom{\lfloor (n-1)/2 \rfloor}{r-1}+1$, then $\cH$ is pancyclic.  
    \end{thm}

    Note that the bound on $\delta(\cH)$ in Theorem~\ref{main} is best possible due to Constructions 1 and 2 above which are not pancyclic because they are nonhamiltonian. On the other hand, we do not expect the bounds on $n$ to be optimal; it can likely be improved with more careful handling.

    {\bf Remark.} Notably, in the cases $r \geq 5$, we present a proof that is independent of Theorem~\ref{KLMthm}. Our proof is shorter than that of Theorem~\ref{KLMthm} and reduces to working in a dense subgraph of the so-called $2$-shadow of $\cH$ in which we apply well known results for nonhamiltonian graphs. We also believe that the same bound on $\delta(\cH)$ should hold for $r \in \{\lfloor(n-1)/2\rfloor, \lfloor (n-1)/2 \rfloor -1\}$, but would require different methods to prove.

    \subsection{Proof outline}

        Let $\cH = (V,E)$ be a hypergraph. The \textbf{$2$-shadow} of $\cH$, denoted $\partial_2(\cH)$, is the graph with $V(\partial_2(\cH)) = V(\cH)$ and $uv \in E(\partial_2(\cH))$ if and only if $\{u,v\} \subseteq e$ for some $e\in E(\cH)$. Equivalently the $2$-shadow of $\cH$ is the graph on $V(\cH)$  in which each $e\in E(\cH)$ induces a clique.
    
        We consider the $2$-shadow of hypergraph $\cH$ and construct a maximal matching between $E(\cH)$ and $E(\partial_2(\cH))$ in which a hyperedge gets matched to a distinct pair of vertices contained in the hyperedge. The pairs that get matched form a subgraph $F$ of $\partial_2(\cH)$. 
    
        In $F$, a cycle of length $\ell$ corresponds to a Berge cycle of length $\ell$ in $\cH$ which is obtained by replacing the cycle edges with their matched hyperedges. Thus if $F$ is pancyclic then so is $\cH$. 
        
        We aim to show that $F$ is pancyclic by proving that it is hamiltonian and weakly pancyclic. Then to show that it has girth $3$ (i.e., contains a triangle) we apply a classical theorem by Mantel.
        \begin{thm}[Mantel \cite{Mantel}]\label{Mantel}
            Let $G$ be an $n$-vertex graph. If $G$ is triangle free, then $|E(G)|\leq \lfloor\frac{n^2}{4}\rfloor$.
        \end{thm}
    
        In Section 3.1, we handle the case $5 \leq r \leq \lfloor (n-1)/2 \rfloor - 2$ by showing that in a nonpancyclic counterexample, the graph $F$ is nonhamiltonian with minimum degree almost $n/2$. We then apply existing stability theorems for such graphs to recover a dense spanning subgraph of $F$ that is {\em almost} hamiltonian. After appropriate small modifications, we can construct a new matching which yields a hamiltonian cycle. We then use Theorem~\ref{Brandt} to show that the resulting graph is weakly pancyclic with girth $3$ thereby  proving that $\cH$ is pancyclic.
    
        In Sections 3.2, we prove Theorem~\ref{main} for the remaining cases of $r$. In particular, we apply Theorem~\ref{KLMthm} and choose a matching that preserves hamiltonicity in $F$. We then show $F$ has sufficiently large minimum degree to apply  Theorem~\ref{Brandt} and conclude $\cH$ is pancyclic. 

\section{Proof of Theorem~\ref{main}}
    Let $3 \leq r \leq \lfloor (n-1)/2 \rfloor - 2$, where $n \geq 70$ if $r\in\{3,4\}$, otherwise $n\geq 39$. Suppose $\cH$ is an  $n$-vertex, $r$-uniform hypergraph with minimum degree $\delta(\cH) \geq \binom{\lfloor (n-1)/2 \rfloor}{r-1} + 1$. For any $v \in V(\cH)$, since $d_{\cH}(v) > (n-1)/(r-1)$, there exists some $u \in V(\cH)$ that is contained in at least $2$ hyperedges of $v$, yielding a Berge cycle of length $2$. It remains to show that $\cH$ contains a Berge cycle of length $\ell$ for each $3 \leq \ell \leq n$. 
    
    Denote by $H=\partial_2(\cH)$ the $2$-shadow of $\cH$. Consider the bipartite graph $B$ with parts $(E(\cH), E(H))$ where $f \in E(\cH)$ is adjacent to $e \in E(H)$ if $e \subset f$.  Suppose $\phi$ is a maximal matching of $B$, and define $\cF = \{f \in E(\cH): f$ is matched in $\phi\}$ and $F = \{\phi(f): f \in \cF\}$ (we view $\cF$ and $F$ as an $r$-graph and a $2$-graph, respectively, both on vertex set $V(\cH)$). Then we can view our matching $\phi$ as a bijective function from the edges of $\cF$ to the edges of $F$. 

    Observe that if $C =v_1,v_2,\ldots, v_\ell, v_1$ is a cycle of length $\ell$ in $F$, then \[\cC=v_1,\phi^{-1}(v_1v_2),v_2, \ldots, v_{\ell-1},\phi^{-1}(v_{\ell-1}v_\ell),v_\ell, \phi^{-1}(v_\ell v_1),v_1\] is a Berge cycle of length $\ell$ in $\cF$. Thus if $F$ is pancyclic then $\cF$ (and therefore $\cH$) is also pancyclic. So let us assume this is not the case and $F$ is not pancyclic.

    \begin{claim}\label{clique}
        If $f \in E(\cH) \setminus E(\cF)$, then the vertices in $f$ induce a clique in $F$.
    \end{claim}
    \begin{proof}
        If some pair of vertices $\{x,y\} \subseteq f$ are nonadajcent in $F$ then we may enlarge $\phi$ to a bigger matching by mapping $f$ to $xy$. This is a contradiction since we assumed $\phi$ to be maximal.
    \end{proof}
    
    \begin{claim}\label{triangle}
        $F$ contains a triangle.
    \end{claim}
    \begin{proof}
       Suppose $F$ does not contain a triangle. If there exists a hyperedge $f \in E(\cH) - E(\cF)$, then by the previous claim, the $r$ vertices in $f$ induce a clique in $F$. Such a clique contains a triangle, so $f$ cannot exist. It follows that every hyperedge of $\cH$ gets matched in $\phi$ to an edge in $F$. Thus
        \[
            \left(\binom{\lfloor (n-1)/2 \rfloor}{r-1}+1\right)n/r \leq \frac{\sum_{v \in V(\cH)} d_\cH(v)}{r} = |E(\cH)| = |E(F)| \leq n^2/4.
        \]
        Where the last inequality follows from Theorem \ref{Mantel}.
    
        We analyze the function $\frac{n}{r}\binom{\lfloor (n-1)/2 \rfloor}{r-1}$ which is smaller than the left hand side.  For $4 \leq r \leq \lfloor (n-1)/2 \rfloor-2$ and $n\geq 11$,
        \[
            \frac{n}{r}\binom{\lfloor (n-1)/2 \rfloor}{r-1} \geq \frac{n}{\lfloor (n-1)/2 \rfloor-2}\binom{\lfloor (n-1)/2 \rfloor}{3}> \frac{n^2}{4}
        \]
        and when $r=3$, we have that
        \[
            \frac{n}{r}\binom{\lfloor (n-1)/2 \rfloor}{r-1} = \frac{n}{3}\binom{\lfloor (n-1)/2 \rfloor}{2}> \frac{n^2}{4},
        \]
        both of which are a contradiction.
    \end{proof}

    We now will use the large minimum degree of $\cH$ to show that $F$ also has large minimum degree, however to do so we consider cases based around the concavity of $\binom{\lfloor  ({n-1})/ 2\rfloor -1}{r-1}$. We first consider $r \geq 5$. Here we use the fact that $\binom{n}{2} \leq \binom{\lfloor (n-1)/2\rfloor-1}{r-2}$ for sufficiently large $n$.

    For a vertex $v \in V(\cH)$, we partition the hyperedges in $\cH$ containing $v$ into two sets $\cN_v$ and $\cL_v$ where $\cN_v = \{f \in E(\cH): v \in f, f \subset \{v\} \cup N_F(v)\}$ is the set of hyperedges containing only $v$ and the $F$ neighborhood of $v$, denoted $N_F(v)$, and $\cL_v = \{f \in E(\cH): v \in f, f \notin \cN_v\}$ is the set of leftover hyperedges. By Claim~\ref{clique} since the $\cL_v$ edges do not induce cliques in $F$,
    \begin{equation}\label{Lv1}\cL_v \subseteq E(\cF).\end{equation} 
        
    \subsection{Proof of Theorem \ref{main} for $5 \leq r\leq \lfloor \frac {n-1}2\rfloor -2$.}
        Suppose $5 \leq r \leq \lfloor \frac {n-1}2\rfloor -2$. 
        \begin{claim}
            $\delta(F) = \lfloor (n-1)/2 \rfloor$ and $F$ is nonhamiltonian.
        \end{claim}
        \begin{proof}
            First we show that $\delta(F) \geq \lfloor (n-1)/2 \rfloor$. Suppose not. Then there exists a vertex $v$ with $d_F(v)\leq \lfloor (n-1)/2 \rfloor - 1$. 
            Then $|\cN_v| \leq \binom{d_F(v)}{r-1}$. On the other hand, 
by~\eqref{Lv1},  $|\cL_v| \leq |E(F)| \leq \binom{n}{2}$.
            Then
            \begin{equation}\label{deg0}
                \binom{\lfloor (n-1)/2\rfloor}{r-1} +1 \leq \delta(\cH) \leq d_{\cH}(v) \leq |\cN_v| + |\cL_v| \leq \binom{\lfloor (n-1)/2 \rfloor -1}{r-1} + \binom{n}{2}.
            \end{equation} 
            For $n \geq 39$, we have $d_{\cH}(v) \leq \binom{n}{2} \leq \binom{\lfloor (n-1)/2\rfloor - 1}{3} \leq \binom{\lfloor (n-1)/2\rfloor - 1}{r-2}$. 
            Thus we have that the right side of~\eqref{deg0} is at most $\binom{\lfloor (n-1)/2\rfloor - 1}{r-1} + \binom{\lfloor (n-1)/2\rfloor - 1}{ r-2} = \binom{\lfloor (n-1)/2 \rfloor}{r-1} < \delta(\cH)$, a contradiction. This shows $\delta(F) \geq \lfloor (n-1)/2 \rfloor$. 
            Then, by Theorem~\ref{Brandt}, $F$ is weakly pancyclic.
        
            Since $F$ has girth $3$ by Claim~\ref{triangle}, if $F$ is hamiltonian, then $F$ is pancyclic, so $\cF$ is pancyclic and we are done. Thus $F$ is nonhamiltonian. Then by Dirac's Theorem, $\delta(F) = \lfloor (n-1)/2 \rfloor$, as desired.
        \end{proof}
    
        We now consider when $F$ is nonhamiltonian. We use a stability result due to Voss~\cite{Voss} (more detailed descriptions by Jung~\cite{Jung} and Jung, Nara~\cite{JungNara} are also available). We only state and use a weaker version. 
        We now define the five classes of graphs $\cG_1, \ldots, \cG_5$ to be used in Voss' stability theorem. See figure \ref{fig:G_i} for a visual reference.
    
        ${}$\enskip --- \enskip The class $\cG_1$ consists of every graph $G$ on $n=2k+2$ vertices with vertex set $V=V_1\cup V_2$ satisfying $|V_1|=|V_2|=k+1$, ($V_1\cap V_2=\emptyset$), and whose edge set is the union of two complete graphs with vertex sets $V_1$ and $V_2$ and at most one further edge $e_0$ (joining $V_1$ and $V_2$).
        \newline
        ${}$\enskip --- \enskip The class $\cG_2$ consists of every graph $G$ on $n=2k+1$ vertices with vertex set $V=V_1\cup V_2$ satisfying $|V_1|=|V_2|=k+1$, $V_1\cap V_2=\{ x_0\}$, and whose edge set is the union of two complete graphs with vertex sets $V_1$ and $V_2$; 
        \newline
        ${}$\enskip --- \enskip The class $\cG_3$ consists of every graph $G$ on $n=2k+2$ vertices with vertex set $V=V_1\cup V_2$ satisfying $|V_1|=k+1$, $|V_2|=k+2$, $V_1\cap V_2=\{ x_0\}$, and whose edge set is the union of a complete graph with vertex set $V_1$ 
        and a $2$-connected graph $G_2$ with vertex set $V_2$ such that $\deg_G(v)\geq k$ for every vertex $v\in V$.
        \newline
        ${}$\enskip --- \enskip The class $\cG_4$ consists of every graph $G$ on $n=2k+1$ vertices with vertex set $V=V_1\cup V_2$ satisfying $|V_1|=k$, $|V_2|=k+1$, ($V_1\cap V_2=\emptyset$), and whose edge set contains all edges joining $V_1$ and $V_2$ and no edges in $V_2$ (so $V_2$ is an independent set).
                \newline
        ${}$\enskip --- \enskip The class $\cG_5$ consists of every graph $G$ on $n=2k+2$ vertices with vertex set $V=V_1\cup V_2$ satisfying $|V_1|=k$, $|V_2|=k+2$, ($V_1\cap V_2=\emptyset$),  in which
        $V_2$ contains at most one edge $e_0$ and $\deg_G(v)\geq k$ for every vertex $v\in V$. In particular, if $e_0$ does not exist then $G$ contains $ K_{k, k+2}$, and if $e_0$ exists, then $G$ contains all but at most $2$ edges between $V_1$ and $V_2$.
        \begin{figure}[h]
            \centering
            \scalebox{0.8}{
            
                \begin{tikzpicture}
                    [shorten >=1pt, minimum size=.01cm, auto,
                    node distance=5cm, thick,
                    main node/.style={circle,inner sep = 2pt, outer sep = 0pt, minimum width = 2pt, fill=black,draw,font=\sffamily\bfseries},
                    main2 node/.style={circle,inner sep = 0.5pt, outer sep = 0pt, minimum width = 2pt, fill=white,font=\sffamily},
                    main3 node/.style={circle,inner sep = 0.5pt, outer sep = 0pt, minimum width = 2pt, fill=white, draw, font=\sffamily}],
                    
                    \draw  (-3,6) circle (1) node {$K_{k+1}$};
                    \draw  (0,6) circle (1) node {$K_{k+1}$};
                    
                    \node [main node] (v1) at (-2,6) {};
                    \node [main node] (v2) at (-1,6) {};
                    \draw [] (v1) -- node[below] {$e_0$} ++ (v2);
        
                    \draw  (3,6) circle (1) node {$K_{k+1}$};
                    \draw  (5,6) circle (1) node {$K_{k+1}$};
                    
                    \node [main node] (v3) at (4,6) {};
                    \node [main2 node] (v3) at (3.65,6.2) {$x_0$};
        
                    \draw  (8,6) circle (1) node {$K_{k+1}$};
                    \draw  (10.2,6) circle (1.2) node {};
                    
                    \node [main node] (v4) at (9,6) {};
                    \node [main2 node] (v4') at (9.35,6.15) {$x_0$};
                    
                    
                    \def\VtwoSpread{3.75, 3, 2.25, 0.25}
                    \draw  (-1,2) ellipse (0.5 and 1.9) node {};
                    \foreach \y / \name in {3.25/1, 2.5/2}{
                        \node [main node] (v\name ) at (-1, \y) {};
                        \draw[ ] (v\name) -- (1.95,1.6) node {};
                        \draw[ ] (v\name) -- (1.95,1.1) node {};
                        \foreach \yprime in \VtwoSpread
                            \draw[] (v\name) -- (2, \yprime) node {};
                    }
                    \draw node at (-1, 1.725) {$\vdots$};
                    \node [main node] (vk) at (-1, .75)  {};
                        \draw[ ] (vk) -- (1.95,1.6) node {};
                        \draw[ ] (vk) -- (1.95,1.1) node {};
                    \foreach \x in \VtwoSpread
                            \draw[] (vk) -- (2, \x) node {};
                    \foreach \y in \VtwoSpread{
                        \node [main node] () at (2, \y) {};
                    }  
                    \draw node at (2, 1.5) {$\vdots$};
                    
                    \draw  (2,2) ellipse (0.5 and 2.2) node[below] {};
                    
                    \def\GfiveVtwoSpread{3.75, 3, 2.25, 0.25}
                    \draw  (5,2) ellipse (0.5 and 1.9) node {};
                    \foreach \y / \name in {3.25/1, 2.5/2}{
                        \node [main node] (v\name ) at (5, \y) {};
                        \draw[ ] (v\name) -- (7.9,1.6) node {};
                        \draw[ ] (v\name) -- (7.9,1.1) node {};
                        \foreach \x in \GfiveVtwoSpread
                            \draw[] (v\name) -- (8, \x) node {};
                    }
                    \draw node at (5, 1.725) {$\vdots$};
                    \node [main node] (vl) at (5, .75)  {};
                    \draw[ ] (vl) -- (7.9,1.6) node {};
                    \draw[ ] (vl) -- (7.9,1.1) node {};
                    \foreach \x in \VtwoSpread
                            \draw[] (vl) -- (8, \x) node {};
                    \foreach \y in \VtwoSpread
                        \node[main node] () at (8, \y) {};
                    \draw  (8,2) ellipse (0.55 and 2.2) node[below] {};
                    \draw[] (8,3) -- (8, 2.25) node [midway, right] {$e_0$};
                    \draw node at (8, 1.5) {$\vdots$};
                    
                \end{tikzpicture}
            }
            \caption{Examples of graphs $G\in \cG_1\cup\cG_2\cup\cG_3\cup\cG_4\cup\cG_5$.}
            \label{fig:G_i}
        \end{figure}
        \begin{lem}[Voss~\cite{Voss}, see~\cite{FKL1} for a short proof]\label{5G}
            Let $k\geq 3$ be an integer, $n \in \{ 2k+1, 2k+2\} $.
            Suppose that $F$ is an $n$-vertex nonhamiltonian graph with $\delta(F) \geq k = \lfloor(n-1)/2\rfloor$.
            Then $F\in \cG_1\cup \dots \cup \cG_5$.  
        \end{lem}
    
        Let $k:= \lfloor (n-1)/2 \rfloor$. The idea for the rest of the proof is as follows. We observe that if $F \in \cG_1 \cup \ldots \cup \cG_5$ then it is {\em almost} hamiltonian---it can be made to be hamiltonian by swapping a small number of its edges with some new edges. Even stronger, it can also be made to be pancyclic in this way. Suppose $F'$ is the graph obtained after swapping such edges. If the edges of $F'$ can also be matched to the hyperedges of $\cH$, then $\cH$ is also pancyclic.
        
        Thus our goal is to show that such an $F'$ exists---we show that our matching $\phi$ can be slightly modified so that the corresponding matched pairs form a pancyclic graph $F'$. 
    
    
        Let $\cA$ and $\cB$ be families of pairs of $V(F)$ with $\cA \subseteq E(F)$. We say $(\cA, \cB)$ is a {\em swappable pair of $F$} if $|\cA| = |\cB|$,  $e_0 \in \cA$ if $e_0$ exists, and

        ${}$\enskip --- \enskip
        if $F\in \cG_1$, $\cB$ consists of exactly two disjoint pairs joining $V_1$ and $V_2$;
        \newline
        ${}$\enskip --- \enskip
        if $F\in \cG_2\cup\cG_3$, $\cB$ consists of exactly one pair $x_1x_2$ joining $V_1\setminus \{ x_0\}$ and $V_2\setminus \{ x_0\}$ (here $x_1\in V_1$ and $x_2 \in V_2$);
        \newline
        ${}$\enskip --- \enskip
        if $F\in \cG_4$, $\cB$ consists of exactly one pair contained in $V_2$;
        \newline
        ${}$\enskip --- \enskip
        and if $F\in \cG_5$, $\cB$ consists of exactly two distinct pairs contained in $V_2$.
    
%
\begin{claim}\label{swap}
    There exists a swappable pair $(\cA, \cB)$ of $F$ such that the graph $F_{\cA,\cB}$ with $E(F_{\cA,\cB}) = (E(F) \setminus \cA) \cup \cB$ is pancyclic and there exists a maximal matching $\phi'$ of $B = (E(\cH), E(H))$ which covers $E(F_{\cA,\cB})$.\footnote{As with $\phi$, this matching should satisfy $e \subset \phi'^{-1}(e)$.}
\end{claim}
        \begin{figure}[ht!]
            \centering
            \scalebox{0.9}{
                \def\ySpread{2, 1.5, 1, 0.5, 0, -.5, -1, -1.5, -2}
                \begin{tikzpicture}
                    [shorten >=1pt, minimum size=.01cm, auto,
                    node distance=5cm, thick,
                    main node/.style={circle,inner sep = 2pt, outer sep = 0pt, minimum width = 2pt, fill=black,draw,font=\sffamily\bfseries},
                    main2 node/.style={circle,inner sep = 0.5pt, outer sep = 0pt, minimum width = 2pt, fill=white,font=\sffamily},
                    main3 node/.style={circle,inner sep = 0.5pt, outer sep = 0pt, minimum width = 2pt, fill=white, draw, font=\sffamily}]
        
                    \draw  (-1,0) circle (1) node {$K_{k+1}$};
                    \draw  (2,0) circle (1) node {$K_{k+1}$};
                    
                    \node [] (x2) at (2,1) {};
                    \node [] (y2) at (2, -1) {};
                    \node [] (y1) at (-1,-1) {};
                    \node [] (x1) at (-1, 1) {};
                    \draw [dashed] (2.75, 0.5) -- (2.75, -0.5);
                    \draw [dashed] (-1.75, 0.55) -- (-1.55, -0.85);
                    \draw [ultra thick] (x2) to (x1);
                    \draw [ultra thick] (y2) to (y1);
                    
                    \def\VtwoSpread{1.9, 1.15, 0.2, -1.9}
                    \draw  (5,0) ellipse (0.5 and 1.9) node {};
                    \foreach \y / \name in {1.25/1, 0.5/2}{
                        \node [main node] (v\name ) at (5, \y) {};
                        \draw[ultra thin] (v\name) -- (7.9, -0.55);
                        \draw[ultra thin] (v\name) -- (7.9, -0.95);
                        
                        \foreach \x in \VtwoSpread
                            \draw[thin] (v\name) -- (8, \x) node {};
                    }
                    \foreach \y in \VtwoSpread{
                        \node[main node] () at (8, \y) {};
                    }
                    \draw node at (5, -0.375) {$\vdots$};
                    \node [main node] (vk) at (5, -1.25)  {};
                    \foreach \x in \VtwoSpread
                        \draw[ultra thin] (vk) -- (8, \x) node {};
                    \draw[ultra thin] (vk) -- (7.9, -0.55);
                    \draw[ultra thin] (vk) -- (7.9, -0.95);
                    \draw node at (8, -0.65) {$\vdots$};
                    \draw (8,0) ellipse (0.5 and 2.2) node[below] {};
                    \draw [color = white] (8,0.2) -- (5, 0.5);
                    \draw [dashed] (8,0.2) -- (5, 0.5);
                    \draw [ultra thick] (8, 1.15) -- (8, 0.2);
                \end{tikzpicture}
            }
            \caption{Examples of the swappable pair $(\cA, \cB)$ for $G\in \cG_1\cup\cG_4$. Dashed edges are elements of $\cA$ and thick edges are elements of $\cB$.}
            \label{fig:swapped}
        \end{figure}
        \begin{proof}
        Recall that for a vertex $v$, $\cL_v$ denotes the set of edges $m\in \cH$ containing a pair $vy$ of $E(H)\setminus E(F)$, and $\cL_v \subseteq \cF$ by~\eqref{Lv1}. If $d_{F}(v)=k$, then $v$ is contained in at most $\binom{k}{r-1} = \binom{\lfloor (n-1)/2 \rfloor}{r-1} < \delta(\cH)$ hyperedges which are subsets of $N_F(v) \cup \{v\}$. Thus we have the following.
        \begin{equation}\label{Lv}
            \hbox{If $d_F(v) =k$, then $\cL_v$ is non-empty.}
        \end{equation}

        We will first construct a swappable pair $(\cA, \cB)$ of $F$, handling the five cases $F \in \cG_1, \ldots, \cG_5$ separately. Then we will show that the resulting $F'$ of swapping $\cA$ with $\cB$ can be matched to hyperedges of $\cH$.

        First suppose $F\in \cG_1$. As in the definition of swappable pairs, we want to choose $\cB$ containing two disjoint pairs between $V_1$ and $V_2$. Define $\cF_{1,2}$ as the set of hyperedges of $\cF$ meeting both $V_1$ and $V_2$. Note that $\cF_{1,2}$ is nonempty by~\eqref{Lv}. Recall that $F$ has at most one edge, $e_0$, joining $V_1$ and $V_2$. If $e_0$ exists, set $m_1 = \phi^{-1}(e_0)$ and choose $x_1x_2=e_0$. Otherwise let $m_1$ be any hyperedge of $\cF_{1,2}$. By symmetry we may suppose that $|m_1\cap V_1|\geq 2$. Choose a vertex $x_2\in V_2\cap m_1$. Since $|m_1| = r \leq k-1$, we can choose another vertex $y_2 \in V_2$ with $y_2 \notin m_1$. Since $y_2$ is not incident to $e_0$ in $F$, $d_F(y_2) = k$. Then we choose an $m_2 \in \cL_{y_2}$. Such a hyperedge must intersect $V_1$ so we can pick a vertex $y_1 \in m_2 \cap V_1$. Lastly, since $|m_1\cap V_1|\geq 2$, we can choose a vertex $x_1 \in m_1 \cap V_1$ such that $x_1 \ne y_1$. We set $(\cA, \cB)= (\{\phi(m_1),\phi(m_2)\},\{x_1x_2,y_1y_2\})$.
    
        Next suppose $F\in \cG_2\cup \cG_3$, and recall $V_1 \cap V_2 = \{x_0\}$. We want $\cB$ to contain a single pair between $V_1$ and $V_2$. Choose a vertex $x_1 \in V_1\setminus \{ x_0\}$  and let $m_1\in \cL_{x_1}$ (which exists by~\eqref{Lv}). Then $m_1$ must intersect $V_2 \setminus \{x_0\}$ since $N_F(x_1) = V_1 \setminus \{x_1\}$. Choose a vertex $x_2 \in m_1 \cap V_2 \setminus \{ x_0\}$. Set $(\cA, \cB) = (\{\phi(m_1)\},\{x_1x_2\})$.
        In the special case when $F\in \cG_3$ and $x_0$ has exactly two $F$-neighbors $x_2$ and some $y_2$ in $V_2$ with $\phi(m_1)=x_0y_2$, we have $\{x_1,x_2,x_0,y_2\} \subseteq m_1$. In this case we can choose the vertex $y_2 \in m_1 \cap V_2 \setminus \{ x_0\}$ to take the place of $x_2$ to ensure $x_0$ always has a $F'$-neighbor other than $x_2$ in $V_2$.
    
        If $F\in \cG_4$ then we want $\cB$ to be a single pair contained in $V_2$. Choose a vertex $x_1 \in V_2$ and let $m_1\in \cL_{x_1}$. Then $m_1$ must contain a vertex $x_2 \in V_2 \setminus \{x_1\}$ since $N_F(x_1) = V_1$. We set $(\cA,\cB)=(\{\phi(m_1)\}, \{x_1x_2\})$.

        Finally, if $F\in \cG_5$ then we want $\cB$ to contain two pairs in $V_2$. Recall from the definition of $\cG_5$, $F$ has at most 1 edge, $e_0$, inside $V_2$.  Set $m_1 = \phi^{-1}(e_0)$ if $e_0$ exists, otherwise let $m_1 \in \cL_v$ for any $v \in V_2$. Thus $m_1$ contains at least two vertices $x_1,x_2$ in $V_2$ (if $m_1 = \phi^{-1}(e_0)$, choose $x_1x_2 = e_0$).  Pick a vertex $y_1 \in V_2\setminus m_1$ and let $m_2$ be an arbitrary member of $\cL_{y_1}$. Such a hyperedge $m_2$ must intersect $V_2$ in at least two vertices so we can choose a vertex $y_2 \in m_2 \cap V_2 \setminus \{y_1\}$. Then set $(\cA, \cB) = (\{\phi(m_1), \phi(m_2)\}, \{x_1x_2, y_1y_2\})$.
    
    
        In each case, we have defined $m_1$ and $x_1x_2$ such that $\{x_1,x_2\} \subset m_1$, and similar for $m_2$ if it exists. We define a new matching $\phi_{\cA,\cB}$ of $B = (E(\cH), E(F_{\cA,\cB}))$ starting from $\phi$ by remapping $m_1$ to $x_1x_2$ and $m_2$ to $y_1y_2$ (if $m_2$ exists). Then $\phi_{\cA,\cB}$ covers $E(F_{\cA,\cB})$, and we may extend it to a maximal matching $\phi'$ if necessary.
        \end{proof}

        Consider the graph $F'$ of the matched edges in $\phi'$. Note that $F_{\cA,\cB} \subseteq F'$ (but it may also contain additional edges if $\phi_{\cA,\cB}$ was not maximal). We will show that $F'$ is pancyclic, and since $\phi'$ matches edges in $F'$ to hyperedges in $\cH$, $\cH$ will also be pancyclic.
        To prove $F'$ is pancyclic, we show that it is hamiltonian and weakly pancyclic with girth $3$. We will use the notion of hamiltonian-closure in our proofs.

        The {\em hamiltonian-closure} of an $n$-vertex graph $G$ is the graph $C(G)$ obtained starting from $G$ by iteratively adding edges between pairs of vertices with degree sum at least $n$. 
        It is well known (see~\cite{BC}) that graph $G$ is hamiltonian if and only if $C(G)$ is hamiltonian.
        We say $G$ is {\em hamiltonian-connected} if for every pair of vertices $x,y \in V(G)$, there is an $x,y$-hamiltonian path in $G$.

        A classical result by P\'osa~\cite{Posa} states that if every pair of non-adjacent vertices $x,y \in V(G)$ satisfy $d_G(u) + d_G(v) \geq n+1$, then $G$ is hamiltonian-connected. In particular, the following Corollary will be useful for us. 
        \begin{cor}\label{cor:hamconn}
        If $e(G)\geq \binom{n}{2}-2$ and $n\geq 6$ then $G$ is hamiltonian-connected.  \qed
        \end{cor}

        \begin{claim} $F'$ is pancyclic.
        \end{claim}
        
        \begin{proof}        
        By construction, $\delta(F')\geq \delta(F)-2 = k-2\geq \frac{n+2}{3}$. By Theorem \ref{Brandt}, $F'$ is weakly pancyclic. By Claim~\ref{triangle} (applied to $F'$ and $\phi'$), $F'$ contains a triangle.  Thus, it suffices to show $F'$ is hamiltonian.
    
        If $F \in \cG_1$, let $\cB=\{x_1x_2,y_1y_2\}$ where $x_1,y_1 \in V_1$ and $x_2,y_2 \in V_2$.  
        We have for $i \in \{1,2\}$, $|E(F'[V_i])|\geq \binom{k+1}{2}-2$.
        By Corollary~\ref{cor:hamconn}, each $V_i$ is hamiltonian-connected and hence there exists an $x_1,y_1$-path $P_1$ and an $x_2,y_2$-path $P_2$ with $V(P_i)=V_i$. Thus $F'$ has a hamiltonian cycle, and we're done.
    
        If $F \in \cG_2 \cup \cG_3$, let $\cB=\{x_1x_2\}$ where $x_1\in V_1 \setminus \{x_0\}$ and $x_2 \in V_2 \setminus \{x_0\}$. We have $|E(F'[V_1])|\geq \binom{k+1}{2}-1$. 
         If there exists an $x_1,x_0$-path $P_1$ and an $x_2,x_0$-path $P_2$ with $V(P_i)=V_i$, then $F'$ has a hamiltonian cycle. By Corollary \ref{cor:hamconn}, such paths exist except possibly for the $x_2,x_0$-path $P_2$ when $F \in \cG_3$. Let $F'_2$ be the graph on $|V_2|+1$ vertices obtained from $F'[V_2]$ by adding a new vertex $x'_2$ and two edges $x_0x'_2$ and $x_2x'_2$. By our choice of vertex $x_2$ stated in the proof of Claim \ref{swap}, there exists a $y_2 \in V_2$ ($y_2 \ne x_2$) such that $x_0y_2 \in E(F'[V_2])$. If $F'_2$ is hamiltonian, then $F'[V_2]$ has a spanning $x_2,x_0$-path $P_2$, and we're done. Otherwise, it suffices to consider the hamiltonian-closure $C(F'_2)$. Since the degrees of $V_2 \setminus \{x_0\}$ in $F'_2$ are at least $k-1$ and $2(k-1) \geq k+3=|V(F'_2)|$, $C(F'_2)$ is a complete graph on $V_2 \setminus \{x_0\}$ and thus $C(F'_2)$ and $F'_2$ are hamiltonian.
    
        In the case $F \in \cG_5$, note that $F$ misses at most $2$ edges between $V_1$ and $V_2$ with equality only if $e_0$ exists. By construction, $F'$ then misses at most 3 edges between $V_1$ and $V_2$.
       Let $\cB=\{x_1x_2,y_1y_2\}$ where $x_1x_2, y_1y_2$ are two distinct edges inside $V_2$. We create a new graph $F'_0$ as follows: Delete the edge $x_1x_2$ and merge $x_1$ and $x_2$ into a new vertex  $x$ with $N_{F'_0}(x)=N_{F'}(x_1) \cap N_{F'}(x_2)$; delete the edge $y_1y_2$ and merge $y_1$ and $y_2$ into a new vertex  $y$ with $N_{F'_0}(y)=N_{F'}(y_1) \cap N_{F'}(y_2)$.
    
        We create another graph $F'_{00}$ from $F'_0$ by joining all $\binom{k}{2}$ pairs in $V_1$ to make $F'_0[V_1]$ a complete graph. If $F'_{00}$ has a hamiltonian cycle then its hamiltonian cycle must only use the edges of $F'_0$ and thus $F'_0$ is hamiltonian as well. Moreover, if $F'_{0}$ contains a hamiltonian cycle, then we can uncontract $x$ and $y$ to get a hamiltonian cycle of $F'$. Hence, it is sufficient to show that $F'_{00}$ is hamiltonian. We consider the hamiltonian-closure $C(F'_{00})$ and observe that for any $u \in F'_{00}[V_1]$ and any $v \in F'_{00}[V_2]$, since $F'_{00}$ misses at most 3 edges between $V_1$ and $V_2$, for $k \geq 5$ we have
        \[d_{F'_{00}}(u)+d_{F'_{00}}(v) \geq 2k-1+k-4=3k-5\geq 2k=|V(F'_{00})|.\]
        Hence, $C(F'_{00})$ contains the complete bipartite graph $K_{k,k}$. It is easy to see that $C(F'_{00})$ is hamiltonian.
        
        The case $F \in \cG_4$ is similar to the $\cG_5$ case but easier so we omit most details. By definition, $F'$ contains an edge $x_1x_2$ in $V_2$. We contract this edge, add edges to make $V_1$ a clique, and take the closure of the resulting graph. As in the $\cG_5$ case, the closure will be hamiltonian and thus $F'$ is hamiltonian. 
        \end{proof}
    
    \subsection{Proof of Theorem~\ref{main} for $r\in \{3,4\}$}

        We now finish the proof of Theorem \ref{main} for the remaining cases.
        Suppose  $r\in \{3,4\}$, and $\cH$ is an $n$-vertex, $r$-uniform hypergraph with $\delta(\cH) \geq \binom{\lfloor (n-1)/2 \rfloor}{r-1} + 1$.  By Theorem \ref{KLMthm}, $\cH$ contains a hamiltonian Berge cycle $\cC$, say $\cC = v_1,e_1,v_2, \ldots, v_n, e_n, v_1$.
    
        Suppose we have chosen $\phi$ to be a maximal matching of $B=(E(\cH), E(H))$ which matches $E(\cH)$ to $E(H)$ such that for $1 \leq i \leq n$, $\phi(e_i) = v_iv_{i+1}$ (with indices taken modulo $n$). As before, we let $F$ be the graph composed of edges in the image of $\phi$. By construction, the image of $\cC$ in $F$ is a hamiltonian cycle $C=v_1,\ldots,v_n,v_1$.
    
        Since $F$ also contains a triangle by Claim \ref{triangle}, if we have $\delta(F) \geq \frac{n+2}{3}$, then $F$ and therefore $\cH$ is pancyclic by applying Theorem \ref{Brandt}. So suppose $\delta(F) < (n+2)/3$. Moreover by Theorem~\ref{Bondy}, since $F \neq K_{n/2, n/2}$,
        \begin{equation}\label{n4}
            |E(F)| < n^2/4.
        \end{equation}
    
        Let $v$ be a vertex with $d_F(v)< \frac{n+2}{3}$. For $r =4$, we have 
        \[\binom{\lfloor (n-1)/2\rfloor}{3} +1 \leq \delta(\cH) \leq d_{\cH}(v) \leq |\cN_v| + |\cL_v| < \binom{\frac{n+2}{3}}{3} + |E(F)| < \binom{\frac{n+2}{3}}{3} + n^2/4.   \]
        This yields a contradiction when $n \geq 3$.
    
    
        %
    
    
        We now complete the proof for $r =3$. Define $S = \{v: d_F(v) \leq (n+1)/3\}$ to be the set of ``small" degree vertices. Then set $S'=\{v: d_{F}(v) \leq \frac{n+1}{3}+1\}$.
    
        Suppose also that we chose $\phi$ to be a maximal matching in which $F$ is hamiltonian satisfying the following: (1) $|S|$ is minimized, and (2) subject to (1), $\sum_{v \in S} d_F(v)$ is maximized.
    
        We will show that if $S$ is non-empty, then we can modify $\phi$ (and therefore $F$) to increase the degree of a vertex in $S$ while leaving the other vertices in $S$ unchanged. This would then contradict the choice of $\phi$.
    
        \begin{claim}
        If $v \in S'$, then $|\cL_v|\geq \frac{1101n^2}{20000}$.
        \end{claim}
    
        \begin{proof}
        If not, then
        \[\binom{\lfloor (n-1)/2\rfloor}{2} +1 \leq \delta(\cH) \leq d_{\cH}(v) \leq |\cN_v| + |\cL_v| < \binom{\frac{n+4}{3}}{2} + \frac{1101n^2}{20000}.\]   
        This yields a contradiction for $n \geq 70$.
        \end{proof}
    
        {\bf Remark}: We chose $\frac{1101}{20000}$ to minimize the bound on $n$ in this claim and the next. One can also use more reasonable constants in exchange for a larger bound on $n$.

        For a hyperedge $\{x,y,z\} \in E(\cF)$, call an edge $xy \in E(F)$ a {\bf private edge} of the vertex $z$ if $\phi(\{x,y,z\})=xy$. Since $r=3$, each $xy \in E(F)$ is a private edge of exactly one vertex. 
    
        \begin{claim}\label{4small}
        If $F$ is not pancyclic, then $|S'|\leq 4$.    
        \end{claim}
    
        \begin{proof}
        Suppose $v_1,v_2,v_3,v_4,v_5 \in S'$. We observe that each vertex $v_i$ has at least $|\cL_{v_i}|-\frac{n+4}{3}$ private edges. Since an edge cannot be a private edge for two different vertices, we have
        \begin{align*}
                5\left(\frac{1101n^2}{20000}-\frac{n+4}{3}\right) & \leq \sum_{i=1}^5\left(|\cL_{v_i}|-\frac{n+4}{3}\right)\\
                & \leq \sum_{i=1}^5 |\{e \in E(F): e \ \mbox{is an private edge of}\ v_i\}| \\
                & \leq |E(F)|.
        \end{align*}   
    
        On the other hand, $|E(F)| < \frac{n^2}{4}$. This yields a contradiction for $n \geq 70$.
        \end{proof}

        Let $v \in S$. A hyperedge $\{u,v,w\}\in \cL_v$ is said to be {\bf shiftable to $v$} if the following holds:
        \begin{enumerate}
            \item $\phi(\{u,v,w\})=uw$ (i.e., $uw$ is a private edge of $v$),
            \item $uw \notin E(C)$, and
            \item $u,w \notin S'$.
        \end{enumerate}
    
        \begin{claim} \label{shiftable hyperedge}
        If $v \in S$, then there always exists a shiftable hyperedge to $v$.
        \end{claim}
    
        \begin{proof}
        Let $v \in S$. We observe that out of the hyperedges in $\cL_v$, at most $d_F(v) \leq \frac{n+1}{3}$ of them do not match to private edges of $v$, at most $n$ of them are matched to the edges on the hamiltonian cycle $C$, and at most $4\left(\frac{n+4}{3}\right)$ of them contain another vertex in $S'$ by Claim~\ref{4small}. Hence, there are always at least 
        \[|\cL_v|-\frac{n+1}{3}-n-4\left(\frac{n+4}{3}\right)\geq \frac{1101n^2}{20000}-\frac{n+1}{3}-n-4\left(\frac{n+4}{3}\right)>0\]  
        shiftable hyperedges to $v$.
        \end{proof}
    
        Take a vertex $v \in S$. By Claim \ref{shiftable hyperedge}, there exists a hyperedge say $\{u,v,w\}\in \cL_v$ that is shiftable to $v$. Without loss of generality, say $vw \notin E(F)$. We define a new matching $\phi'$ of $B$ starting from $\phi$ by remapping $\phi'(\{u,v,w\}) = vw$. In the corresponding graph $F'$ of the image of $\phi'$, we have that $d_{F'}(v) = d_F(v) + 1$, $d_{F'}(u) = d_F(u) - 1 > (n+1)/3$, and $d_{F'}(x) = d_F(x)$ for all other vertices $x \in V(F')$. Moreover, $F'$ still contains the hamiltonian cycle $C$. We can then extend $\phi'$, if necessary, to a maximal matching of $B$. This contradicts the choice of $\phi$ by rules (1) and (2). \qed

\section{Concluding Remarks}

    \begin{enumerate}

        \item Our main theorem holds for uniformities $3 \leq r \leq \lfloor (n-1)/2 \rfloor -2$. We believe the same minimum degree conditions also imply pancyclicity for $r \in \{\lfloor (n-1)/2 \rfloor -1, \lfloor (n-1)/2 \rfloor\}$, however our current proof breaks down for these cases because the minimum degree $\binom {\lfloor (n-1)/2 \rfloor}{r-1} $ rapidly decreases as $r$ increases. For $r \geq n/2$, we believe the correct minimum degree should be $\delta(\cH) \geq r$ as in Theorem~\ref{KLMthm}.

        \item The main case in our proof of Theorem~\ref{main} in which we consider $r\geq 5$ is independent of Theorem~\ref{KLMthm} and reduces to analyzing dense nonhamiltonian graphs. It would be interesting if the remaining cases could also be proved without applying Theorem~\ref{KLMthm}. 
        \item It seems that a very similar proof can be used to show the minimum degree in Theorem~\ref{FKLthm} is sufficient for pancyclicity in non-uniform hypergraphs.
        \item The ``exceptional graph" in Theorem~\ref{Bondy}, $K_{n/2, n/2}$ is hamiltonian but not pancyclic. Clearly, bipartite graphs are not pancyclic because they avoid all odd cycles. In contrast, the sharpness examples for Theorem~\ref{main}, Constructions 1 and 2, are not pancyclic hypergraphs because they do not contain hamiltonian Berge cycles. 
It may be interesting to ask {\em how dense can an $r$-uniform hamiltonian hypergraph be if it is not pancyclic?} This question can be phrased either in terms of the number of edges or the minimum degree.

In regards to this question, the authors previously believed it would be difficult to construct {\em any} $r$-uniform hypergraph that is hamiltonian but not pancyclic for $r\geq 3$. Notably we cannot find an easy generalization of bipartite graphs. For $r=3$, we give the following example on $2k$ vertices that does not contain a Berge cycle of any length between $4$ and $k-1$ (this can be shown using case analysis).

        {\bf Construction 3.} For an integer $k\geq 5$, let $\cH_{2k} = (V,E)$ where $V = \{0, 1, \ldots, 2k-1\}$ and $E = \{\{i, i+1, i+k\} : i \in V\}$ where our vertices are taken modulo $2k$.
        
        For reference $\cH_{12}$ is included below:
    
        \begin{center}
            \includegraphics[scale=.45]{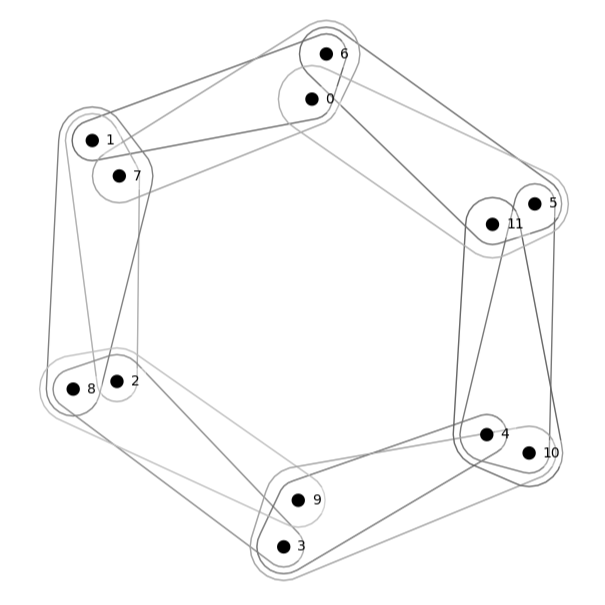}
        \end{center}
    
    
    
    
            

This example does not seem to generalize to $r \geq 4$.
However,  recently Nika Salia suggested a family of constructions for all small uniformities.

{\bf Construction 4.}  Start with a (graph) cycle $v_1, v_2, \ldots, v_k, v_1$ and for $1 \leq i \leq k$ replace edge $v_iv_{i+1}$ with a complete $r$-uniform hypergraph on $r+1$ vertices $K_{r+1}^{(r)}$ containing $\{v_i,v_{i+1}\}$ and $r-1$ distsinct new vertices (here indices are taken modulo $k$).  Let $\cH_{k,r}$ be the resulting hypergraph.

The hypergraph $\cH_{k,r}$ has $n=kr$ vertices and contains a hamiltonian Berge cycle. For $k$ large compared to $r$, its $2$-shadow contains no cycles with length between $r+2$ and $k-1 = n/r - 1$. Hence $\cH_{k,r}$ does not contain Berge cycles of these lengths. (One can also check that $\cH_{k,r}$ contains Berge cycles of all other lengths.)

We pose the question: do hamiltonian $r$-uniform hypergraphs that are not pancyclic exist for larger $r$? Say $r \geq n/2$?

    \end{enumerate}

{\bf Acknowledgements.} The authors would like to thank Nika Salia for his helpful discussions. The authors are also grateful to the anonymous referees for their many useful suggestions.

\printbibliography

\end{document}